\documentclass[preprint,3p,times]{elsarticle}

\usepackage{amssymb}
\usepackage{amsmath}
\usepackage{graphicx}
\usepackage{amsfonts}
\usepackage{amsthm}
\usepackage{mathrsfs}
\usepackage[title]{appendix}
\usepackage{xcolor}
\usepackage{enumitem}
\usepackage{tikz}
\usepackage[colorlinks=true]{hyperref}
\usepackage[all,cmtip]{xy}
\input diagxy

\newcommand{\CL}{\mathcal{L}}

\newcommand{\CS}{\mathcal{S}}
\newcommand{\CT}{\mathcal{T}}

\newcommand{\ra}{\rightarrow}

\newcommand{\bbN}{\mathbb{N}}

\newcommand{\bbR}{\mathbb{R}}

\newcommand{\tauTL}{\tau_{_{T,L}}}
\newcommand{\Delp}{\Delta^+}
\newcommand{\tauTSL}{\tau_{_{T^*,L}}}
\newcommand{\vep}{\varepsilon}
\renewcommand{\leq}{\leqslant}
\renewcommand{\geq}{\geqslant}

\numberwithin{equation}{section}

\allowdisplaybreaks

\theoremstyle{plain}
  \newtheorem{thm}{Theorem}[section]
  \newtheorem{lem}[thm]{Lemma}
  \newtheorem{prop}[thm]{Proposition}
  
\theoremstyle{definition}

  \newtheorem{rem}[thm]{Remark}

  \newtheorem{prob}[thm]{Problem}

\makeatletter
\def\ps@pprintTitle{%
 \let\@oddhead\@empty
 \let\@evenhead\@empty
 \def\@oddfoot{\centerline{\thepage}}%
 \let\@evenfoot\@oddfoot}
\makeatother

\begin{document}

\begin{frontmatter}

\title{Generalized supremum convolutions on the tensor product of unit intervals: t-conorms and triangle functions}

\author{Mengyu Luo}\ead{lmysrem@qq.com}

\author{Lili Shen}\ead{shenlili@scu.edu.cn}

\author{Jie Zhang\corref{cor}}\ead{zjandzxz@qq.com}

\cortext[cor]{Corresponding author.}

\affiliation{organization={School of Mathematics, Sichuan University},
            city={Chengdu 610064},
            country={China}}

\begin{abstract}
The complete-lattice tensor product $\mathbb{E}=[0,1]\otimes[0,1]$ arises naturally in the algebra of truth values for type-2 fuzzy sets. For a monotone surjective map $S$ that is continuous on $[0,1[^2$ and a monotone map $T$ having $1$ as identity, we consider the generalized supremum convolution
\[
(f\otimes_{_{S,T}}g)(x)=\sup_{S(u,v)=x}T(f(u),g(v))
\]
on $\mathbb{E}$. Rather than assuming classical algebraic structures for $S$ and $T$ from the outset, we ask when they are forced by requiring the induced convolution itself to be a t-conorm. We prove that $\otimes_{_{S,T}}$ is a t-conorm on $\mathbb{E}$ if and only if $S$ is a continuous t-conorm and either $S=\max$ or $T$ is a right-continuous t-norm; in the exceptional case $S=\max$, the convolution is simply the pointwise maximum for every admissible $T$. The same convolution mechanism also yields an order-dual description of triangle functions on the space of distance distribution functions $\Delta^+$. Using an order-reversing isomorphism between $\Delta^+$ and a boundary variant of $\mathbb{E}$, we show that $\tau_{_{T^*,L}}$ is a triangle function precisely when $L$ is a continuous t-conorm with no zero divisor and, unless $L=\max$, $T^*$ is a left-continuous t-conorm; moreover, when $L\neq\max$, continuity of $\tau_{_{T^*,L}}$ is equivalent to continuity of $T^*$. This resolves the remaining $\tau_{_{T^*,L}}$ part of the Schweizer--Sklar problem.
\end{abstract}


\end{frontmatter}

\section{Introduction}

Type-2 fuzzy sets, introduced by Zadeh \cite{Zadeh1975} as an extension of ordinary fuzzy sets, take their membership grades in the function space
\[
\mathbb{M}=[0,1]^{[0,1]}.
\]
Convolution operations on this truth-value space and on its structured subspaces have played a central role in the algebraic study of type-2 fuzzy sets \cite{Walker2005,Harding2008,Harding2010,DeMiguel2018}. In the present paper we focus on the complete lattice $\mathbb{E}$ of decreasing functions $f\colon[0,1]\to[0,1]$ satisfying $f(0)=1$ and left-continuity on $]0,1]$. This lattice identifies with the tensor product
\[
\mathbb{E}=[0,1]\otimes[0,1]
\]
in the category $\mathbf{Sup}$ of complete lattices and sup-preserving maps
\cite{GutierrezGarcia2017a,Eklund2018}.

A standard way to lift operations on $[0,1]$ to fuzzy truth values is Zadeh's extension principle. For a binary operation $*$ on $[0,1]$, the classical sup-min form assigns to $f,g\in\mathbb{M}$ the function
\[
x\mapsto \sup_{u*v=x}\min\{f(u),g(v)\}.
\]
Various generalized extensions and convolution constructions have subsequently been studied in the theory of type-2 fuzzy sets \cite{Kawaguchi2002,Hernandez2015,Wu2020,Wu2022}. In particular, the literature contains both sufficient-condition results for generalized extended t-norms and converse results asking whether the algebraic axioms of the induced operation force corresponding properties of the generating operation \cite{Hernandez2015,Wu2020}. Recent work has further clarified the relationships among several notions of type-2 t-norm arising from such generalized extensions \cite{Zhang2025a}. Related convolution structures on function spaces over bounded lattices were studied in \cite{DeMiguel2018}. These results motivate a broader converse question: how much structure of the generating operations is forced by the requirement that the induced convolution itself be a t-conorm?

To address this question on $\mathbb{E}$, we start with a map $S\in\CS$, assumed only to be surjective, increasing in each place, and continuous on $[0,1[^2$, and a map $T\in\CT$, assumed only to be increasing in each place and to have $1$ as identity. We then define
\[
(f\otimes_{_{S,T}}g)(x)
=
\sup_{S(u,v)=x}T(f(u),g(v)),
\qquad f,g\in\mathbb{E}.
\]
Thus neither a t-conorm structure on $S$ nor a t-norm structure on $T$ is assumed a priori. Our main theorem (Theorem \ref{thm:tensor_tconorm_characterization}) gives an exact characterization: $\otimes_{_{S,T}}$ is a well-defined t-conorm on $\mathbb{E}$ if and only if $S$ is a continuous t-conorm and either $S=\max$ or $T$ is a right-continuous t-norm. The exceptional case $S=\max$ is rigid in a different way: the convolution collapses to pointwise maximum and is therefore independent of the choice of $T\in\CT$. Hence the algebraic axioms of the induced operation, together with the closure conditions defining $\mathbb{E}$, recover the classical structure of the generating operations except in this precisely identified degenerate case. This distinguishes the present setting from earlier results on normal convex fuzzy truth values, where the operation determining the constraint in the convolution is already taken to be a t-norm or t-conorm \cite{Hernandez2015,Wu2020}.

The same convolution mechanism reappears, in order-dual form, in the theory of probabilistic metric spaces. Schweizer and Sklar asked for necessary and sufficient conditions under which the operations $\tauTL$ and $\tauTSL$ define (continuous) triangle functions on the space $\Delta^+$ of distance distribution functions \cite{Schweizer1983}; triangle functions and their generalized convolution forms are further discussed in \cite{SamingerPlatz2008}. The $\tauTL$ part of this problem was resolved in \cite{Lai2026}. We show that the remaining $\tauTSL$ problem is not an unrelated application of the preceding theory: an order-reversing isomorphism identifies $\Delta^+$ with a boundary variant $\mathbb{E}'$ of $\mathbb{E}$ and transforms $\tauTSL$ exactly into the supremum convolution studied above. This reduces the triangle-function problem to the same structural mechanism, with one additional boundary condition. As a consequence, $\tauTSL$ is a triangle function if and only if $L$ is a continuous t-conorm with no zero divisor and, unless $L=\max$, $T^*$ is a left-continuous t-conorm (Theorem \ref{thm:triangle_function_characterization}); moreover, $\tauTSL$ is continuous if and only if, in the nonmaximum case, $T^*$ is continuous (Theorem \ref{thm:continuous_triangle_function_characterization}). Together with \cite{Lai2026}, this completes the corresponding Schweizer--Sklar characterization for these two convolution-type triangle functions.

The content of this paper is organized as follows. Section \ref{Preliminaries} recalls the tensor product $\mathbb{E}$, the classes $\CS$ and $\CT$, and the required facts about t-norms and t-conorms. Section \ref{T-conorms-on-E} establishes the characterization theorem for $\otimes_{_{S,T}}$. Section \ref{Application-Schweizer-Sklar} constructs the order-reversing identification with $\Delta^+$ and derives the characterizations of triangle functions and continuous triangle functions.

\section{Preliminaries} \label{Preliminaries}

The terminology and notation used in this paper are primarily drawn from the literature on complete lattices, t-norms, and probabilistic metric spaces \cite{Klement2000,Schweizer1983}. Throughout this paper, we write
\begin{itemize}
\item $]a,b[=\{x\in\bbR\mid a<x<b\}$,
\item $[a,b[=\{x\in\bbR\mid a\leq x<b\}$,
\item $]a, b]=\{x\in\bbR\mid a<x\leq b\}$ and
\item $[a,b]=\{x\in\bbR\mid a\leq x\leq b\}$
\end{itemize}
for intervals in $\bbR$.


The \emph{tensor product} of the complete lattice $[0,1]$ with itself,
denoted by
\[
\mathbb{E}:=[0,1]\otimes[0,1],
\]
can be represented as the complete lattice consisting of functions
$f\colon[0,1]\to[0,1]$ satisfying
\[
f\left(\sup A\right)=\inf f(A)
\]
for all subsets $A\subseteq[0,1]$
\cite{GutierrezGarcia2017a,Eklund2018}.
In other words, $f\colon [0,1]\to[0,1]$ belongs to $\mathbb{E}$ if 
\begin{itemize}
\item $f(0)=1$, 
\item $f$ is decreasing on $[0,1]$, and
\item $f$ is left-continuous on $]0,1]$.
\end{itemize}
The order on $\mathbb{E}$ is defined pointwise.

In particular, for each $r\in[0,1]$, the function
\[\omega_r(x)=\begin{cases}
            1 & \text{if}\ 0\leq x \leq r,\\
            0 & \text{if}\ r < x \leq 1
\end{cases}\]
is in $\mathbb{E}$, and it will be frequently used later.


\begin{rem}
It is worth noting the connection between $\mathbb{E}$ and the truth value algebra of type-2 fuzzy sets; related convolution algebras and lattice structures have been studied in \cite{Walker2005,Harding2008,Harding2010,DeMiguel2018}. Let 
\[\mathbb{M} = [0,1]^{[0,1]}\] 
denote the set of all such truth values. The standard meet operation $\sqcap$ on $\mathbb{M}$, derived via Zadeh's extension principle, is given by
\[(f \sqcap g)(x) = \sup\limits_{\min\{r,s\} = x} \min \{f(r), g(s)\}\]
for all $f,g \in \mathbb{M}$ and $x \in [0,1]$. This operation naturally induces a meet order $\sqsubseteq_\sqcap$ on $\mathbb{M}$ defined by 
\[f \sqsubseteq_\sqcap g\iff f \sqcap g = f.\]
Since the functions in $\mathbb{E}$ are decreasing, \cite[Proposition 2.4.4]{Harding2016} guarantees that the meet order $\sqsubseteq_\sqcap$ restricted to $\mathbb{E}$ coincides precisely with the usual pointwise order. Consequently, $\mathbb{E}$ becomes a complete lattice under this order.
\end{rem}

\begin{lem} \label{finE:upper-semi-continuous}
Every function $f\in\mathbb{E}$ is upper semi-continuous.
\end{lem}

\begin{proof}
Since $f \in \mathbb{E}$, it is decreasing and left-continuous. Let $u\in[0,1]$ and $u_n\ra u$. If $u=0$, then
\[\limsup_{n\ra\infty}f(u_n)\leq 1=f(0),\]
so the assertion is clear. Assume now that $u>0$. For any $\varepsilon>0$, by the left-continuity of $f$ at $u$ there exists $\delta>0$ such that
\[u-\delta < t \leq u \implies f(t) < f(u)+\varepsilon.\]
Choose an integer $N$ such that $u_n>u-\delta$ for all $n\geq N$. If $u_n\leq u$, then $f(u_n)<f(u)+\varepsilon$; if $u_n>u$, then $f(u_n)\leq f(u)$ because $f$ is decreasing. In either case, we have $\limsup\limits_{n\ra\infty} f(u_n)\leq f(u)+\varepsilon$, and consequently $\limsup\limits_{n\ra\infty} f(u_n)\leq f(u)$ by the arbitrariness of $\varepsilon$. Thus, $f$ is upper semi-continuous.
\end{proof}

We denote by $\CT$ \cite{Schweizer1983} the set of binary operations $T$ on $[0,1]$ that satisfy the following conditions:
\begin{enumerate}[label=(T\arabic*)]
\item \label{T1} $T$ is increasing in each place, i.e., $T(x_1, y_1) \leq T(x_2, y_2)$ whenever $x_1 \leq x_2$ and $y_1 \leq y_2$, and 
\item \label{T2} $T$ has $1$ as identity, i.e., $T(1,x) = T(x,1) = x$ for all $x\in [0,1]$.
\end{enumerate} 
Moreover, we denote by $\CS$ the set of binary operations $S$ on $[0,1]$ that satisfy the following conditions:
\begin{enumerate}[label=(S\arabic*)]
\item \label{S1} $S$ is surjective, 
\item \label{S2} $S$ is increasing in each place, and 
\item \label{S3} $S$ is continuous on $[0,1[^2$.
\end{enumerate}
 
For any $T \in \CT$ and $S \in \CS$, there is a map 
\[-\otimes_{_{S,T}}-\colon\mathbb{E} \times \mathbb{E}\to\mathbb{M}\]
with
\[(f \otimes_{_{S,T}} g)(x) = \sup\limits_{S(u,v)=x} T(f(u),g(v))\]
for all $f, g \in \mathbb{E}$ and $x \in [0,1]$.

\begin{rem} \label{rem:continuity_E}
The purpose of this paper is to characterize the necessary and sufficient conditions for $\otimes_{_{S,T}}$ to be a t-conorm on $\mathbb{E}$. For this construction to be meaningful, $\otimes_{_{S,T}}$ must first be a well-defined binary operation on $\mathbb{E}$. If $S$ lacks sufficient continuity, the supremum convolution can easily produce discontinuities that violate the decreasing property of the resulting function, causing it to fall outside $\mathbb{E}$.

To illustrate the potential failure of this operation without such continuity, consider a discontinuous $S$ defined by 
\[S(x,0)=x,\quad S(0,y)=y\quad\text{and}\quad S(x,y)=1\] 
for $x,y > 0$. Taking $f=g=\omega_{0.5} \in \mathbb{E}$, a direct calculation yields
\[(f \otimes_{_{S,T}} g)(1) = 1, \quad \text{whereas for } x \in [0.8, 1[, \ (f \otimes_{_{S,T}} g)(x) = 0.\]
Since $(f \otimes_{_{S,T}} g)(0.8) < (f \otimes_{_{S,T}} g)(1)$, the result fails to be a decreasing function.

Therefore, we assume condition \ref{S3} to facilitate our discussion. This choice avoids assuming continuity on the whole square $[0,1]^2$ while providing sufficient regularity for the arguments below.
\end{rem}

Let $(P, \leq, 0, 1)$ be a bounded poset. A \emph{t-norm} on $P$ \cite{DeBaets1999} is a commutative and associative binary operation  $T\colon P^2\to P$ that has $1$ as identity and is increasing in each place. A \emph{t-conorm} on $P$ is a commutative and associative binary operation  $S:P^2 \to P$ that has $0$ as identity and is increasing in each place.

\begin{rem}
Since the primary objective of this paper is to establish the necessary and sufficient conditions for $\otimes_{_{S,T}}$ to be a t-conorm on $\mathbb{E}$, it is instructive to explicitly instantiate the above definition for our underlying lattice. Recall that the minimum element in $\mathbb{E}$ is $\omega_0$. Consequently, a t-conorm on $\mathbb{E}$ is precisely a commutative and associative binary operation on $\mathbb{E}$ that is increasing in each place and possesses $\omega_0$ as its identity element.
\end{rem}

Let $T$ be a t-norm on the closed interval $[0,1]$. We say that $T$ is \emph{continuous} if it is a continuous binary function on $[0,1]^2$. $T$ is \emph{left-continuous} if
\[T(x,y)=\sup \, \{T(u,v) \mid u<x,\ v<y\}\]
for all $x, y \in ]0,1]$. Similarly, $T$ is \emph{right-continuous} if
\[T(x,y)=\inf \, \{T(u,v) \mid u > x,\ v > y\}\]
for all $x, y \in [0,1[$. For an increasing binary operation $T$, this condition is equivalent to right-continuity in each variable separately; that is,
\[
T(x,y)=\inf\,\{T(u,y)\mid u>x\}
      =\inf\,\{T(x,v)\mid v>y\}
\]
for all $x,y\in[0,1[$. It is well known that $T$ is continuous if and only if it is both left-continuous and right-continuous.

\begin{lem} \label{lem:max(x,y)-leq-S(x,y)}
If an operation $S\in\CS$ has $0$ as identity, then $S(x,y)\geq\max\{x,y\}$ for all $x,y\in[0,1]$. In particular, $S(x,1)=S(1,y)=1$ for all $x,y\in[0,1]$.
\end{lem}

\begin{proof}
Since $0$ is the identity of $S$, we have $x = S(x,0) \leq S(x,y) \leq 1$, and similarly $y = S(0,y) \leq S(x,y) \leq 1$ for all $x,y \in [0,1]$. Hence, $\max\{x,y\} \leq S(x,y) \leq 1$ for all $x,y \in [0,1]$. In particular, $1=\max\{x,1\}\leq S(x,1)\leq 1$ yields $S(x,1)=1$, and similarly $S(1,y)=1$.
\end{proof}

\begin{lem} \label{lem:s1_implies_cont} 
If an operation $S\in\CS$ has $0$ as identity, then it is continuous on the entire domain $[0,1]^2$.
\end{lem}

\begin{proof}
Since $S \in \CS$, the condition \ref{S3} guarantees that $S$ is continuous on $[0,1[ \times [0,1[$. It remains to verify its continuity at the boundaries, i.e., at points $(x_0, 1)$ and $(1, y_0)$ for any $x_0, y_0 \in [0,1]$.

Since $0$ is the identity of $S$, we have $\max\{x,y\}\leq S(x,y)\leq 1$ by Lemma \ref{lem:max(x,y)-leq-S(x,y)}. Let $(x,y)$ approach $(x_0, 1)$ in $[0,1]^2$. Since $\lim\limits_{(x,y) \ra (x_0, 1)} \max\,\{x,y\} = \max\,\{x_0, 1\} = 1$, the squeeze theorem yields
\[\lim\limits_{(x,y) \ra (x_0, 1)} S(x,y) = 1.\]
Furthermore, Lemma \ref{lem:max(x,y)-leq-S(x,y)} ensures that $S(x_0, 1) = 1$. Thus, $\lim\limits_{(x,y) \ra (x_0, 1)} S(x,y) = S(x_0, 1)$, proving that $S$ is continuous at $(x_0, 1)$. It follows similarly that $S$ is continuous at $(1, y_0)$. Therefore, $S$ is continuous on the entire domain $[0,1]^2$.
\end{proof}

Let $S$ be a continuous t-conorm on $[0,1]$. An element $x \in [0,1]$ is said to be \emph{idempotent} if $S(x,x)=x$. Clearly, the set of all idempotent elements of a continuous t-conorm forms a closed subset of $[0,1]$ containing $0$ and $1$. Thus, for any non-idempotent element $x\in]0,1[$, we can naturally define the following two elements:
\[a_x = \max\, \{e \in [0, x] \mid S(e,e)=e \}, \quad b_x = \min\, \{e \in [x, 1] \mid S(e,e)=e \}.\]
Restricting the binary operation $S$ to the closed interval $[a_x,b_x]$, we obtain an ordered monoid, denoted by 
\[([a_x,b_x],S).\]

Before proceeding, recall that the truncated addition (or bounded sum) on $[0,1]$, denoted by $\oplus$, is defined as $u \oplus v = \min\,\{u+v, 1\}$ for all $u,v \in [0,1]$. The following conclusion is of fundamental importance in the theory of continuous t-conorms. It follows straightforwardly from \cite[Corollary 5.12]{Klement2000}.

\begin{prop}\label{ordinal_sum}
Let $S$ be a continuous t-conorm on $[0,1]$. If $x\in]0, 1[$ is non-idempotent, then the restricted monoid $([a_x,b_x],S)$ is isomorphic, as an ordered monoid, to either $([0,\infty],+)$ or $([0,1],{\oplus})$. In particular, if $S$ has no non-trivial idempotent elements (i.e., only $0$ and $1$ are idempotent), then $([0,1],S)$ is isomorphic to either $([0,\infty],+)$ or $([0,1],{\oplus})$.
\end{prop}

\section{T-conorms on the tensor product $\mathbb{E}$} \label{T-conorms-on-E}

The primary objective of this section is to characterize the necessary and sufficient conditions under which the operation $\otimes_{_{S,T}}$ is a well-defined t-conorm on the tensor product space $\mathbb{E}$. We first establish its basic monotonicity and boundary behavior, and then prove that it preserves the defining decreasing and left-continuous properties of $\mathbb{E}$. We next verify the required algebraic properties and derive the necessary conditions on the generating operations $S$ and $T$. These results lead to the main characterization theorem of the section.

\begin{prop}\label{prop:tensorST_decreasing} 
For any $T \in \CT$ and any $S \in \CS$ that has $0$ as its identity, the function $f \otimes_{_{S,T}} g$ is decreasing on $[0,1]$ and satisfies $(f \otimes_{_{S,T}} g)(0)=1$. Moreover, the map $\otimes_{_{S,T}}$ is increasing in each place on $\mathbb{E} \times \mathbb{E}$.
\end{prop}

\begin{proof}
First, we show that $f \otimes_{_{S,T}} g$ is decreasing on $[0,1]$. Let $x_1, x_2 \in [0,1]$ with $x_1 < x_2$. For any pair $(u_2, v_2)$ such that $S(u_2, v_2) = x_2$, we observe that since $S$ has $0$ as its identity, Lemma \ref{lem:s1_implies_cont} ensures that its range on the rectangle $[0, u_2] \times [0, v_2]$ covers the interval $[0, x_2]$\footnote{Indeed, the rectangle $[0,u_2]\times[0,v_2]$ is connected, and by Lemma \ref{lem:s1_implies_cont} the restriction of $S$ to this rectangle is continuous. Hence its image is a connected subset of $[0,1]$. Since this image contains both $S(0,0)=0$ and $S(u_2,v_2)=x_2$, it must contain the whole interval $[0,x_2]$.}. Because $x_1<x_2$, there exists $(u_1,v_1)\in[0,u_2]\times[0,v_2]$ such that $S(u_1,v_1)=x_1$.

Since $f, g \in \mathbb{E}$ are decreasing, it follows that $f(u_1) \geq f(u_2)$ and $g(v_1) \geq g(v_2)$. The monotonicity of $T$ then implies
\[
    T(f(u_1), g(v_1)) \geq T(f(u_2), g(v_2)).
\]
By the definition of the supremum, we have
\[
    (f \otimes_{_{S,T}} g)(x_1) \geq T(f(u_1), g(v_1)) \geq T(f(u_2), g(v_2)).
\]
Taking the supremum over all pairs $(u_2, v_2)$ satisfying $S(u_2, v_2) = x_2$ yields
\[
    (f \otimes_{_{S,T}} g)(x_1) \geq \sup\limits_{S(u_2,v_2)=x_2} T(f(u_2), g(v_2)) = (f \otimes_{_{S,T}} g)(x_2).
\]
Thus, $f \otimes_{_{S,T}} g$ is decreasing.

Second, we evaluate $(f \otimes_{_{S,T}} g)(0)$. By definition,
\[
    (f \otimes_{_{S,T}} g)(0) = \sup\limits_{S(u,v)=0} T(f(u), g(v)).
\]
Since $S$ is increasing in each place and has $0$ as its identity, we have $S(u,v) \geq S(u, 0) = u$ and $S(u,v) \geq S(0, v) = v$. Therefore, the condition $S(u,v) = 0$ implies $u = 0$ and $v = 0$. Since $f(0) = g(0) = 1$ for all functions in $\mathbb{E}$, and $T$ has $1$ as its identity, we obtain
\[
    (f \otimes_{_{S,T}} g)(0) = T(f(0), g(0)) = T(1, 1) = 1.
\]

Finally, we show that $\otimes_{_{S,T}}$ is increasing in each place on $\mathbb{E} \times \mathbb{E}$. Let $f_1, f_2 \in \mathbb{E}$ such that $f_1 \leq f_2$. For any $g \in \mathbb{E}$ and any $(u,v) \in [0,1]^2$, the monotonicity of $T$ guarantees that $T(f_1(u), g(v)) \leq T(f_2(u), g(v))$. Taking the supremum over the set $\{(u,v) \mid S(u,v) = x\}$, we deduce
\[(f_1\otimes_{_{S,T}} g)(x) = \sup\limits_{S(u,v)=x} T(f_1(u), g(v)) \leq \sup\limits_{S(u,v)=x} T(f_2(u), g(v)) = (f_2\otimes_{_{S,T}} g)(x).\]
This confirms that $f_1 \otimes_{_{S,T}} g \leq f_2 \otimes_{_{S,T}} g$. By the symmetry of the conditions on the arguments, $\otimes_{_{S,T}}$ is also increasing in its second place.
\end{proof}

\begin{lem}\label{lem:supremum_attained}
Suppose that $T$ is a right-continuous t-norm on $[0,1]$ and $S$ is a continuous t-conorm on $[0,1]$. Given $f,g\in\mathbb{E}$ and $x\in [0,1]$, there exists a pair $(u_x, v_x)$ such that 
\[S(u_x, v_x) = x\quad\text{and}\quad(f\otimes_{_{S,T}} g)(x) = T (f(u_x),g(v_x)).\]
\end{lem}
\begin{proof}
Let $K_x = \{(u,v) \in [0,1]^2 \mid S(u,v) = x\}$. Since $S$ is a t-conorm, we have $S(u,v) \geq \max\,\{u,v\}$ by Lemma \ref{lem:max(x,y)-leq-S(x,y)}. Thus, any $(u,v) \in K_x$ must satisfy $u \leq x$ and $v \leq x$, implying that $K_x \subseteq [0,x]^2$. By the continuity of $S$, $K_x$ is a closed subset of $[0,1]^2$ and is therefore compact.

Next, we show that the composite function 
\[h\colon [0,1]^2\to[0,1],\quad h(u,v) := T(f(u),g(v))\]
is upper semi-continuous. Let $(u_n, v_n)$ be an arbitrary sequence converging to $(u,v)$. By the upper semi-continuity of $f$ and $g$ (Lemma \ref{finE:upper-semi-continuous}), for any $\varepsilon > 0$, there exists an integer $N$ such that $f(u_n) < f(u) + \varepsilon$ and $g(v_n) < g(v) + \varepsilon$ for all $n \geq N$. The monotonicity of the t-norm $T$ then yields
\[
    T(f(u_n), g(v_n)) \leq T(\min\{f(u)+\varepsilon, 1\}, \min\{g(v)+\varepsilon, 1\}).
\]
Taking the limit superior as $n\ra\infty$ on the left side we obtain
\[
    \limsup\limits_{n\ra\infty} T(f(u_n), g(v_n)) \leq T(\min\,\{f(u)+\varepsilon, 1\}, \min\,\{g(v)+\varepsilon, 1\}).
\]
Crucially, since $T$ is right-continuous, taking the limit as $\varepsilon \ra 0^+$ on the right side gives
\[
    \limsup\limits_{n\ra\infty} T(f(u_n), g(v_n)) \leq T(f(u), g(v)).
\]
This confirms that $h$ is upper semi-continuous.

Finally, by the generalized extreme value theorem (Weierstrass theorem), an upper semi-continuous function is guaranteed to attain its maximum on a compact set. Therefore, there exists a point $(u_x, v_x) \in K_x$ such that
\[
    \sup\limits_{S(u,v)=x} T(f(u),g(v)) = \max_{(u,v) \in K_x} h(u,v) = T(f(u_x),g(v_x)).
\]
This completes the proof.
\end{proof}

\begin{prop} \label{prop:binary_operation}
Given a continuous t-conorm $S$ on $[0,1]$ and a right-continuous t-norm $T$ on $[0,1]$, $\otimes_{_{S,T}}$ is a binary operation on $\mathbb{E}$.
\end{prop}

\begin{proof} 
Given $f,g \in \mathbb{E}$, we need to show that $f \otimes_{_{S,T}} g \in \mathbb{E}$. First, by Proposition \ref{prop:tensorST_decreasing}, $f \otimes_{_{S,T}} g$ is decreasing and $(f \otimes_{_{S,T}} g)(0) = 1$. It remains to show that $f \otimes_{_{S,T}} g$ is left-continuous on $]0,1]$.

Suppose, for the sake of contradiction, that $f \otimes_{_{S,T}} g$ is not left-continuous at some $x^* \in ]0,1]$. Because $f \otimes_{_{S,T}} g$ is decreasing, there must exist $\eta > 0$ and a strictly increasing sequence $(x_n)_{n\in\bbN}$ converging to $x^*$ such that, for all $n \in \bbN$,
\[
    (f \otimes_{_{S,T}} g)(x^*) \leq (f \otimes_{_{S,T}} g)(x_n) - \eta.
\]
By Lemma \ref{lem:supremum_attained}, for each $n \in \bbN$, there is a pair $(u_n, v_n)$ with $S(u_n, v_n) = x_n$ such that $(f \otimes_{_{S,T}} g)(x_n) = T(f(u_n), g(v_n))$. Hence, for all $n\in \bbN$,
\begin{equation}\label{eq:f_tensor_g_bound}
     (f \otimes_{_{S,T}} g)(x^*) \leq T(f(u_n), g(v_n)) - \eta.
\end{equation}

Since $x_n < x^* \leq 1$ and $\max\{u_n, v_n\} \leq S(u_n, v_n) = x_n$, the sequences $(u_n)_{n\in\bbN}$ and $(v_n)_{n\in\bbN}$ are bounded within the compact square $[0, x^*]^2$. Thus, we can iteratively extract monotone subsequences $(u_{n_k})_{k\in\bbN}$ and $(v_{n_k})_{k\in\bbN}$ converging to $u^*$ and $v^*$, respectively\footnote{More explicitly, by compactness we first pass to a subsequence, still denoted by $(u_n,v_n)$, converging to some $(u^*,v^*)\in[0,x^*]^2$. By the monotone subsequence theorem, the sequence $(u_n)$ has a monotone subsequence. Passing to this subsequence, the corresponding sequence of second coordinates is still convergent to $v^*$; applying the monotone subsequence theorem once more, we obtain a further subsequence along which the second coordinate sequence is monotone. Along this final subsequence, both coordinate sequences are monotone and still converge to $u^*$ and $v^*$, respectively. As usual, we relabel this final subsequence as $(u_{n_k},v_{n_k})$.}. Since $S(u_{n_k}, v_{n_k}) = x_{n_k}$, the continuity of $S$ yields $S(u^*,v^*) = x^*$.

We now distinguish three cases based on the monotonicity of the subsequences:

\emph{Case 1}. Both sequences $(u_{n_k})_{k\in\bbN}$ and $(v_{n_k})_{k\in\bbN}$ increase to $u^*$ and $v^*$, respectively. Since $f$ and $g$ are decreasing and left-continuous, their values $f(u_{n_k})$ and $g(v_{n_k})$ decrease to $f(u^*)$ and $g(v^*)$, respectively. By the right-continuity of $T$, passing to the limit as $k \ra \infty$ in \eqref{eq:f_tensor_g_bound} yields
\begin{align*}
     (f \otimes_{_{S,T}} g)(x^*) 
     &\leq \lim\limits_{k\ra\infty} T(f(u_{n_k}), g(v_{n_k})) - \eta \\
     &= T(f(u^*), g(v^*)) - \eta \\
     &\leq \sup\limits_{S(u,v)=x^*} T(f(u),g(v)) - \eta \\
     &= (f \otimes_{_{S,T}} g)(x^*) - \eta,
\end{align*}
which implies $0 \leq -\eta$, a contradiction.

\emph{Case 2}. Both sequences $(u_{n_k})_{k\in\bbN}$ and $(v_{n_k})_{k\in\bbN}$ decrease to $u^*$ and $v^*$, respectively. If this were the case, the monotonicity of $S$ would imply $x_{n_k} = S(u_{n_k}, v_{n_k}) \geq S(u^*, v^*) = x^*$, which contradicts the fact that $x_{n_k} < x^*$.

\emph{Case 3}. Without loss of generality, assume that the sequence $(u_{n_k})_{k\in\bbN}$ increases to $u^*$ and the sequence $(v_{n_k})_{k\in\bbN}$ decreases to $v^*$. Since $f$ is decreasing and left-continuous, $f(u_{n_k})$ decreases to $f(u^*)$. Since $g$ is decreasing and $v_{n_k} \geq v^*$, we have $g(v_{n_k}) \leq g(v^*)$ for all $k$. Applying these facts to \eqref{eq:f_tensor_g_bound}, we obtain
\[
    (f \otimes_{_{S,T}} g)(x^*) \leq T(f(u_{n_k}), g(v_{n_k})) - \eta \leq T(f(u_{n_k}), g(v^*)) - \eta
\]
for all $k$. Taking the limit as $k \ra \infty$, the right-continuity of $T$ with respect to its first argument ensures that $T(f(u_{n_k}), g(v^*))$ converges to $T(f(u^*), g(v^*))$. Therefore,
\begin{align*}
    (f \otimes_{_{S,T}} g)(x^*) 
    &\leq \lim\limits_{k\ra\infty} T(f(u_{n_k}), g(v^*)) - \eta \\
    &= T(f(u^*), g(v^*)) - \eta \\
    &\leq \sup\limits_{S(u,v)=x^*} T(f(u),g(v)) - \eta \\
    &= (f \otimes_{_{S,T}} g)(x^*) - \eta,
\end{align*}
which again leads to a contradiction.

Since all possible cases lead to a contradiction, $f \otimes_{_{S,T}} g$ must be left-continuous. In conclusion, $\otimes_{_{S,T}}$ is a well-defined binary operation on $\mathbb{E}$.
\end{proof}

\begin{rem}
The monotone-subsequence argument in the proof of Proposition \ref{prop:binary_operation} closely parallels the argument used in \cite[Lemma~8.1]{SamingerPlatz2008}. The proof above is adapted to the decreasing, left-continuous functions in $\mathbb{E}$ and to the supremum convolution considered here.
\end{rem}

\begin{prop} \label{prop:is_tconorm}
   Given a continuous t-conorm $S$ on $[0,1]$ and a right-continuous t-norm $T$ on $[0,1]$, $\otimes_{_{S,T}}$ is a t-conorm on $\mathbb{E}$.
\end{prop}
\begin{proof}
    By Proposition \ref{prop:binary_operation}, $\otimes_{_{S,T}}$ is a well-defined binary operation on $\mathbb{E}$. Given $f\in \mathbb{E}$ and $x\in [0,1]$, we evaluate its identity element:
 \[
 (f \otimes_{_{S,T}} \omega_0)(x) = \sup\limits_{S(u,v)=x}T(f(u),\omega_0(v)).
 \]
 For any $v > 0$, we have $\omega_0(v)=0$, which implies $T(f(u),\omega_0(v)) = T(f(u), 0) = 0$ (because $0$ is the annihilator of the t-norm $T$). Since the supremum is bounded below by $0$, we only need to consider the case where $v=0$. Since $S(u,0) = u$, the condition $S(u,v)=x$ forces $u=x$. Thus, we have 
 \[
 (f \otimes_{_{S,T}} \omega_0)(x) = T(f(x), \omega_0(0)) = T(f(x), 1) = f(x).
 \]
 Hence, $f \otimes_{_{S,T}} \omega_0 = f$. The commutativity of $\otimes_{_{S,T}}$ follows immediately from the commutativity of $S$ and $T$, which also implies $\omega_0 \otimes_{_{S,T}} f = f$. Therefore, $\omega_0$ is the identity of $\otimes_{_{S,T}}$. Furthermore, recall from Proposition \ref{prop:tensorST_decreasing} that $\otimes_{_{S,T}}$ is increasing in each place. 
 
 It remains to show that $\otimes_{_{S,T}}$ is associative. Given $f, g, h \in \mathbb{E}$, we need to prove that
 \[
     (f \otimes_{_{S,T}} g) \otimes_{_{S,T}} h = f \otimes_{_{S,T}} (g \otimes_{_{S,T}} h).
 \]
 Since $\otimes_{_{S,T}}$ is a binary operation on $\mathbb{E}$, both sides of the equation are functions in $\mathbb{E}$, meaning they both evaluate to $1$ at $x=0$. For $x \in ]0,1]$, the left-hand side is given by
\[
     ((f \otimes_{_{S,T}} g) \otimes_{_{S,T}} h)(x) 
     = 
     \sup\limits_{S(u,v) = x} T((f \otimes_{_{S,T}} g)(u), h(v))
     = \sup\limits_{S(u,v) = x} T\left(\sup\limits_{S(r,s) = u}T(f(r),g(s)), h(v)\right).
\]
 Our goal is to verify the following equality:
 \[
    \sup\limits_{S(u,v) = x} T\left(\sup\limits_{S(r,s) = u}T(f(r),g(s)), h(v)\right) = \sup\limits_{S(S(a,b),c) = x}T(T(f(a),g(b)),h(c)).
 \]
 First, we establish the inequality ``$\geq$''. For any $a,b,c\in [0,1]$ satisfying $S(S(a,b),c) = x$, let $u = S(a,b)$ and $v = c$. Then $S(u,v)=x$, and by the definition of the supremum, we have
 \[
    \sup\limits_{S(r,s) = u}T(f(r),g(s)) \geq T(f(a), g(b)).
 \]
 By the monotonicity of $T$, it follows that
 \[
    T\left(\sup\limits_{S(r,s) = u}T(f(r),g(s)), h(v)\right) \geq T(T(f(a),g(b)),h(c)).
 \]
 Taking the supremum over all such triplets $(a,b,c)$ yields
 \[
    \sup\limits_{S(u,v) = x} T\left(\sup\limits_{S(r,s) = u}T(f(r),g(s)), h(v)\right) \geq \sup\limits_{S(S(a,b),c) = x}T(T(f(a),g(b)),h(c)).
 \]
 
 Next, we establish the reverse inequality ``$\leq$''. By Lemma \ref{lem:supremum_attained}, there exist $u_0, v_0 \in [0,1]$ with $S(u_0, v_0) = x$ such that
 \[
    ((f \otimes_{_{S,T}} g) \otimes_{_{S,T}} h)(x) = T((f \otimes_{_{S,T}} g)(u_0), h(v_0)).
 \]
 Applying Lemma \ref{lem:supremum_attained} again to the inner operation, there exist $r_0, s_0 \in [0,1]$ with $S(r_0, s_0) = u_0$ such that
 \[
    (f \otimes_{_{S,T}} g)(u_0) = T(f(r_0), g(s_0)).
 \]
 Thus, we have $S(S(r_0, s_0), v_0) = x$, which allows us to deduce
 \[
     ((f \otimes_{_{S,T}} g) \otimes_{_{S,T}} h)(x) 
     = 
     T(T(f(r_0), g(s_0)), h(v_0))
     \leq 
     \sup\limits_{S(S(a,b),c) = x}T(T(f(a),g(b)),h(c)).
 \]
 Therefore, we conclude that for all $x\in [0,1]$,
 \[
    ((f \otimes_{_{S,T}} g) \otimes_{_{S,T}} h)(x) = \sup\limits_{S(S(a,b),c) = x}T(T(f(a),g(b)),h(c)).
 \]
 A completely symmetric argument shows that the right-hand side of the associativity equation is
 \[
    (f \otimes_{_{S,T}} (g \otimes_{_{S,T}} h))(x) = \sup\limits_{S(a,S(b,c)) = x}T(f(a),T(g(b),h(c))).
 \]
 By the associativity of $S$ and $T$, the bounds of the suprema and the inner expressions are identical, meaning the right-hand sides of these two expressions are equal. Consequently,
 \[
    (f \otimes_{_{S,T}} g) \otimes_{_{S,T}} h = f \otimes_{_{S,T}} (g \otimes_{_{S,T}} h).
 \]
 This concludes the proof.
\end{proof}

\begin{prop} \label{prop:s_is_continuous_tconorm}
Given $T \in \CT$ and $S \in \CS$, if $\otimes_{_{S,T}}$ is a t-conorm on $\mathbb{E}$, then $S$ is a continuous t-conorm.
\end{prop}
\begin{proof} 
Our primary strategy is to show that for all $u,v \in [0,1]$, 
\[
    {\omega_u} \otimes_{_{S,T}} {\omega_v} = \omega_{S(u,v)}.
\]
We evaluate the function $\omega_u \otimes_{_{S,T}} \omega_v$ point by point. First, consider the point $x = S(u,v)$. Since the pair $(u,v)$ trivially satisfies the supremum condition $S(r,s) = S(u,v)$, we have
\[
    (\omega_u \otimes_{_{S,T}} \omega_v)(S(u,v)) \geq T(\omega_u(u), \omega_v(v)) = T(1,1) = 1.
\]
Because the codomain is $[0,1]$, it follows that $(\omega_u \otimes_{_{S,T}} \omega_v)(S(u,v)) = 1$. 

Next, since $\otimes_{_{S,T}}$ is assumed to be a t-conorm on $\mathbb{E}$, the function $\omega_u \otimes_{_{S,T}} \omega_v$ belongs to $\mathbb{E}$ and is therefore decreasing. Hence, for all $x \in [0, S(u,v)]$,
\[
    (\omega_u \otimes_{_{S,T}} \omega_v)(x) \geq (\omega_u \otimes_{_{S,T}} \omega_v)(S(u,v)) = 1,
\]
which implies $(\omega_u \otimes_{_{S,T}} \omega_v)(x) = 1$.

For $x > S(u,v)$, any pair $(r,s)$ satisfying $S(r,s) = x$ must have either $r > u$ or $s > v$. Indeed, if both $r \leq u$ and $s \leq v$ held, the monotonicity of $S$ would imply $S(r,s) \leq S(u,v) < x$, a contradiction. Thus, for any such pair $(r,s)$, we have either $\omega_u(r) = 0$ or $\omega_v(s) = 0$. Since $T \in \CT$ is increasing and has $1$ as its identity, we know $T(1,0) = T(0,1) = 0$. By monotonicity, $T(a,b) = 0$ whenever $a=0$ or $b=0$. Consequently, for all $x > S(u,v)$, the supremum is taken over a set of zeros, yielding
\[
    (\omega_u \otimes_{_{S,T}} \omega_v)(x) = 0.
\]

Combining these cases, the function is exactly $1$ on $[0, S(u,v)]$ and $0$ on $]S(u,v), 1]$; this also covers the endpoint cases $S(u,v)=0$ and $S(u,v)=1$. Thus
\[
    \omega_u \otimes_{_{S,T}} \omega_v = \omega_{S(u,v)}.
\]
 
With this identity established, the properties of $S$ follow directly from those of $\otimes_{_{S,T}}$. Because $\omega_0$ is the identity of the t-conorm $\otimes_{_{S,T}}$, we have for all $v \in [0,1]$,
\[
    \omega_v = \omega_0 \otimes_{_{S,T}} \omega_v = \omega_{S(0,v)},
\]
which implies $S(0,v) = v$, meaning $0$ is the identity of $S$. The commutativity and associativity of $\otimes_{_{S,T}}$ similarly enforce that $\omega_{S(u,v)} = \omega_{S(v,u)}$ and $\omega_{S(S(u,v),w)} = \omega_{S(u,S(v,w))}$, which imply that $S$ is commutative and associative. 

Since $S \in \CS$ is already known to be increasing and surjective, it satisfies all the axioms of a t-conorm on $[0,1]$. Finally, by Lemma \ref{lem:s1_implies_cont}, any t-conorm belonging to $\CS$ is necessarily continuous on the entire domain $[0,1]^2$.
\end{proof}

\begin{prop}\label{prop:t_is_tnorm}
Given $T \in \CT$ and $S \in \CS$, if $\otimes_{_{S,T}}$ is a t-conorm on $\mathbb{E}$ and $S \neq \max$, then $T$ is a t-norm on $[0,1]$.
\end{prop}
\begin{proof}
By Proposition \ref{prop:s_is_continuous_tconorm}, $S$ is a continuous t-conorm. Since $S \neq \max$, there exists a non-idempotent element $z \in ]0,1[$. By Proposition \ref{ordinal_sum}, the restricted monoid $([a_z,b_z],S)$ is isomorphic to either $([0,\infty],+)$ or $([0,1],{\oplus})$. Because these Archimedean components allow elements to be strictly increased without immediately reaching the upper bound, we can choose $c \in ]a_z, b_z[$ sufficiently close to $a_z$ such that $S(S(c,c),c) < b_z$. Since $a_z$ is the identity of this restricted monoid and $c$ is non-idempotent, it naturally holds that $c < S(c,c) < S(S(c,c),c) < b_z \leq 1$.

Given arbitrary $p,q,r \in [0,1]$, we define three functions in $\mathbb{E}$:
\[f(x)= \begin{cases} 1, & x=0, \\ p, & 0 < x \leq c, \\ 0, & c < x \leq 1, \end{cases} \quad g(x)= \begin{cases} 1, & x=0, \\ q, & 0 < x \leq c, \\ 0, & c < x \leq 1, \end{cases} \quad h(x)= \begin{cases} 1, & x=0, \\ r, & 0 < x \leq c, \\ 0, & c < x \leq 1. \end{cases}\]
First, we investigate the point values of the function $f \otimes_{_{S,T}} g$.

It is clear that $(f \otimes_{_{S,T}} g)(0)=1$ since $\otimes_{_{S,T}}$ is a well-defined t-conorm on $\mathbb{E}$. For $0 < x \leq c$, the condition $\max\{u, v \} \leq S(u,v) = x$ implies that $u \leq x \leq c$ and $v \leq x \leq c$. Since $S(0,0) = 0 < x$, we must have either $u > 0$ or $v > 0$. Hence, the possible values of $T(f(u),g(v))$ are restricted to $T(1,q)$, $T(p,1)$, and $T(p,q)$. Since $T \in \CT$ has $1$ as its identity by definition, $T(1,q)=q$ and $T(p,1)=p$. Taking into account the monotonicity of $T$, one concludes that
\begin{align*}
(f \otimes_{_{S,T}} g)(x) &= \sup\limits_{S(u,v) = x} T(f(u),g(v))\\ 
&= \max \, \{ T(1,q) , T(p,1) , T(p,q) \} \\ 
&= \max \, \{ q, p , T(p,q) \} \\ 
&= \max \, \{p,q\}. 
\end{align*}

For $c < x \leq S(c,c)$, if $u=0$, then $v=x>c$, which implies $g(v)=0$ and hence $T(f(u),g(v))=T(1,0)=0$. Symmetrically, if $v=0$, then $T(f(u),g(v))=0$. By the continuity of $S$ and the fact that $S(0,c) = c < x \leq S(c,c)$, there must exist $u_0, v_0 \in ]0,c]$ such that $S(u_0,v_0)=x$. For these strictly positive values bounded by $c$, we have $T(f(u_0),g(v_0))=T(p,q)$. Thus,
\[(f \otimes_{_{S,T}} g)(x) = \sup\limits_{S(u,v) = x} T(f(u),g(v)) = T(f(u_0),g(v_0)) = T(p,q).\]

For $x = S(u,v) > S(c,c)$, the monotonicity of $S$ dictates that either $u > c$ or $v > c$. For any such pair $(u,v)$, we inevitably have $f(u)=0$ or $g(v)=0$, either of which forces $T(f(u),g(v))=0$. Consequently, $(f \otimes_{_{S,T}} g)(x)=0$. 

In summary, the composite function is perfectly determined:
\[(f \otimes_{_{S,T}} g)(x) = \begin{cases} 1, & x = 0, \\ \max \, \{p,q\}, & 0 < x \leq c, \\ T(p,q), & c < x \leq S(c,c),\\ 0, & x > S(c,c). \end{cases}\]
By symmetry, interchanging $f$ and $g$ yields:
\[(g\otimes_{_{S,T}} f)(x) = \begin{cases} 1, & x = 0, \\ \max\,\{q,p\}, & 0 < x \leq c, \\ T(q,p), & c < x \leq S(c,c),\\ 0, & x > S(c,c). \end{cases}\]

Since $\otimes_{_{S,T}}$ is assumed to be a t-conorm, it is commutative, meaning $f \otimes_{_{S,T}} g = g \otimes_{_{S,T}} f$. Evaluating both functions on the non-empty interval $]c, S(c,c)]$, we immediately deduce that $T(p,q) = T(q,p)$. 

To prove the associativity of $T$, we compare
\[
((f\otimes_{_{S,T}}g)\otimes_{_{S,T}}h)(x)
\quad\text{and}\quad
(f\otimes_{_{S,T}}(g\otimes_{_{S,T}}h))(x).
\]
Because we explicitly ensured $S(c,c) < S(S(c,c),c)$, the interval $]S(c,c), S(S(c,c),c)]$ is strictly non-empty. Let $x$ belong to this interval. For $((f\otimes_{_{S,T}}g)\otimes_{_{S,T}}h)(x)$, if a pair $(u,v)$ with $S(u,v)=x$ satisfied $u\leq c$ and $v\leq c$, then $S(u,v)\leq S(c,c)<x$, a contradiction. If $v>c$, then $h(v)=0$, while if $u>S(c,c)$, then $(f\otimes_{_{S,T}}g)(u)=0$ by the formula above. Moreover, if $v=0$ and $u\leq S(c,c)$, then $S(u,v)=u\leq S(c,c)<x$, which is impossible. Hence a nonzero contribution is possible only when $c<u\leq S(c,c)$ and $0<v\leq c$, where
\[
(f\otimes_{_{S,T}}g)(u)=T(p,q)
\quad\text{and}\quad
h(v)=r.
\]
Such pairs exist: since
\[
S(S(c,c),0)=S(c,c)<x\leq S(S(c,c),c),
\]
the continuity of $v\mapsto S(S(c,c),v)$ yields some $v\in]0,c]$ with $S(S(c,c),v)=x$. All remaining pairs yield zero, so
\[
((f\otimes_{_{S,T}}g)\otimes_{_{S,T}}h)(x)=T(T(p,q),r).
\]
The same argument, with the first and second inner convolutions interchanged, gives
\[
(f\otimes_{_{S,T}}(g\otimes_{_{S,T}}h))(x)=T(p,T(q,r)).
\]
The associativity of $\otimes_{_{S,T}}$ therefore yields $T(T(p,q),r)=T(p,T(q,r))$. Given the arbitrariness of $p, q, r \in [0,1]$, we conclude that $T$ is both commutative and associative. Coupled with its inherent identity and monotonicity from $\CT$, $T$ is a well-defined t-norm on $[0,1]$.
\end{proof}

\begin{prop}\label{prop:t_is_rightcont}
Given $T \in \CT$ and $S \in \CS$, if $\otimes_{_{S,T}}$ is a t-conorm on $\mathbb{E}$ and $S \neq \max$, then $T$ is right-continuous.
\end{prop}
\begin{proof}
Since $T$ is increasing in each place, it suffices to establish that
\[T(x_0, y_0) \geq \inf\,\{ T(x, y) \mid x > x_0,\ y > y_0 \}\]
for any given $x_0, y_0 \in [0,1[$.

Since $S$ is a continuous t-conorm and $S \neq \max$, there exists a non-idempotent element $z \in ]0,1[$. By Proposition \ref{ordinal_sum}, the restricted monoid $([a_z,b_z],S)$ is isomorphic to either $([0,\infty],+)$ or $([0,1],{\oplus})$. We can choose $c \in ]a_z, b_z[$ close enough to $a_z$ such that $S(c,c) < b_z$. Because both standard addition $+$ and truncated addition $\oplus$ are strictly increasing on intervals where their sums do not hit the upper bound, the operation $S$ is strictly increasing on the restricted domain $[a_z, c]^2$.

Now, we define two functions $f, g \in \mathbb{E}$ as follows:
\[f(x) = \begin{cases} 
1, & x=0,\\
x_0 + (1-x_0)\frac{c-x}{c}, & 0 < x \leq c,\\
0, & c < x \leq 1, 
\end{cases}
\quad
g(x) = \begin{cases} 
1, & x=0,\\
y_0 + (1-y_0)\frac{c-x}{c}, & 0 < x \leq c,\\
0, & c < x \leq 1. 
\end{cases}\]
It is routine to verify that $f, g \in \mathbb{E}$ since they are decreasing, satisfy $f(0)=g(0)=1$, and are left-continuous on $]0,1]$.

Consider the value of $f \otimes_{_{S,T}} g$ at the specific point $x = S(c,c)$. For any pair $(u,v)$ satisfying $S(u,v) = S(c,c)$, if $u > c$ or $v > c$, then $f(u)=0$ or $g(v)=0$, either of which yields $T(f(u), g(v)) = 0$. Hence, to find the supremum, we only need to consider the region where $u \leq c$ and $v \leq c$. For such $u, v$:
\begin{itemize}
\item If $u \leq a_z$, since $a_z$ is the identity of the restricted monoid, we have $S(u,v) \leq S(a_z,c) = c < S(c,c)$. Similarly, if $v\leq a_z$, then $S(u,v)\leq S(c,a_z)=c<S(c,c)$.
\item If $a_z < u \leq c$ and $a_z < v \leq c$, the strict monotonicity of $S$ on $[a_z,c]^2$ guarantees that $S(u,v) = S(c,c)$ if and only if $u=c$ and $v=c$.
\end{itemize}
Thus $(c,c)$ is the only point of the fibre $S^{-1}(S(c,c))$ contained in $[0,c]^2$. Every other point of this fibre has $u>c$ or $v>c$ and hence contributes $0$. Consequently,
\[(f \otimes_{_{S,T}} g)(S(c,c)) = \sup\limits_{S(u,v) = S(c,c)} T(f(u), g(v)) = T(f(c),g(c)) = T(x_0,y_0).\]

Next, let $x\in]0,S(c,c)[$. To ensure strict inequalities, we observe the diagonal map $t \mapsto S(t,t)$. Since $S$ is continuous, $S(0,0)=0$, and $x < S(c,c)$, the intermediate value theorem guarantees the existence of some $t_x \in ]0, c[$ such that $S(t_x, t_x) = x$. Setting $u_x = v_x = t_x$, we strictly have $u_x < c$ and $v_x < c$. For these values, $f(u_x) > x_0$ and $g(v_x) > y_0$ hold due to the strict monotonicity of $f$ and $g$ on $]0,c]$. It follows that:
\[(f \otimes_{_{S,T}} g)(x) = \sup\limits_{S(u,v)=x} T(f(u),g(v)) \geq T(f(u_x),g(v_x)) \geq \inf \, \{ T(\alpha, \beta) \mid \alpha > x_0, \beta > y_0 \}.\]

Crucially, since $\otimes_{_{S,T}}$ is a t-conorm on $\mathbb{E}$, the resulting function $f \otimes_{_{S,T}} g$ must belong to $\mathbb{E}$ and is therefore left-continuous at $S(c,c)$. Taking the limit as $x$ approaches $S(c,c)$ from the left ($x \to S(c,c)^-$), we obtain:
\[(f \otimes_{_{S,T}} g)(S(c,c)) = \lim\limits_{x \ra S(c,c)^-} (f \otimes_{_{S,T}} g)(x) \geq \inf \, \{ T(\alpha, \beta) \mid \alpha > x_0, \beta > y_0 \}.\]
Substituting the point value calculated earlier, we conclude:
\[T(x_0,y_0) \geq \inf \, \{ T(\alpha, \beta) \mid \alpha > x_0, \beta > y_0 \}.\]
This explicitly demonstrates that $T$ is right-continuous in the sense defined above.
\end{proof}

Summarizing the preceding results, we arrive at the central characterization of this paper:

\begin{thm} \label{thm:tensor_tconorm_characterization} 
Given $T \in \CT$ and $S \in \CS$, the map $\otimes_{_{S,T}}$ is a t-conorm on $\mathbb{E}$ if and only if the following conditions hold:
\begin{enumerate}[label={\rm(\alph*)}]
\item \label{thm:tensor_tconorm_characterization:S} $S$ is a continuous t-conorm;
\item \label{thm:tensor_tconorm_characterization:T} $T$ is a right-continuous t-norm on $[0,1]$ if $S \neq \max$.
\end{enumerate}
\end{thm}
\begin{proof}
If $S\neq\max$, the sufficiency of the conditions follows directly from Propositions \ref{prop:binary_operation} and \ref{prop:is_tconorm}. If $S=\max$, then for any $T \in \CT$, given $f,g \in \mathbb{E}$ and $x\in [0, 1]$,
\begin{align*} 
(f \otimes_{_{\max,T}} g)(x) &= \sup\limits_{\max \{u,v \} = x}T (f(u), g(v))\\ 
&= \max \, \{ T (f(0), g(x)), T (f(x), g(0)) \} & (f,g\ \text{are decreasing})\\ 
&= \max \, \{ f(x), g(x) \}. 
\end{align*}
Since the pointwise maximum operation preserves the decreasing and left-continuous properties of $\mathbb{E}$ and satisfies all axioms of a t-conorm with $\omega_0$ as its neutral element, $\otimes_{_{\max,T}}$ is a t-conorm on $\mathbb{E}$. Thus the conditions are sufficient in all cases.

Conversely, the necessity of condition \ref{thm:tensor_tconorm_characterization:S} is established in Proposition \ref{prop:s_is_continuous_tconorm}, and the necessity of condition \ref{thm:tensor_tconorm_characterization:T} is collectively provided by Propositions \ref{prop:t_is_tnorm} and \ref{prop:t_is_rightcont}.
\end{proof}

\section{Application: Resolving the Schweizer-Sklar open problem} \label{Application-Schweizer-Sklar}

In this section, we provide a complete answer to the open problem posed by Schweizer and Sklar regarding the triangle function $\tauTSL$ on $\Delta^+$. We show that the properties of $\tauTSL$ on $\Delta^+$ can be characterized through operations on the tensor product space, using the framework established in Section \ref{T-conorms-on-E} while carefully accounting for the specific boundary behavior of distance distribution functions.

A \emph{distance distribution function} \cite{Schweizer1983} is a map 
\[\phi\colon [0,\infty]\to[0,1]\] 
such that 
\begin{itemize}
\item $\phi(0)=0$, $\phi(\infty)=1$, 
\item $\phi$ is increasing, and 
\item $\phi$ is left-continuous on $]0,\infty[$.
\end{itemize}
Let $\Delta^+$ denote the set of all distance distribution functions. The partial order on $\Delta^+$ is defined pointwise.

For the purpose of defining the metric, every $f\in\Delta^+$ is extended to $]-\infty,0[$ by setting $f(x)=0$ for all $x<0$.

Let $f,g \in \Delta^+$, $t \in ]0,1]$, and let $(f,g;t)$ denote the condition 
\[
f(x-t)-t \leq g(x) \leq f(x+t)+t, \quad  \forall x \in \Big[0,\frac{1}{t}\Big[.
\]
The \emph{modified L\'{e}vy metric} \cite{Schweizer1983} is the function $d_L$ defined on $\Delta^+ \times \Delta^+$ by 
\[
d_L(f,g)= \inf \, \{\,t \mid \text{both } (f,g;t) \text{ and } (g,f;t) \text{ hold}\}.
\]
Here the subscript in $d_L$ refers to L\'{e}vy and is unrelated to the operation $L$ introduced below.
A sequence $\{f_n\}$ in $\Delta^+$ is said to \emph{converge weakly} to $f \in \Delta^+$ (denoted by $f_n \xrightarrow{w} f$) if the sequence $\{f_n\}$ converges to $f$ at each continuity point $x$ of $f$.

\begin{prop} (See \cite{Schweizer1983}.) \label{prop:weakly_converge}
    Let $\{f_n\}$ be a sequence of functions in $\Delta^+$, and let $f \in \Delta^+$. Then $f_n \xrightarrow{w} f$ if and only if $d_L(f_n, f) \rightarrow 0$.
\end{prop}

A \emph{triangle function} \cite{Schweizer1983} $\tau$ is a binary operation on $\Delta^+$ that is commutative, associative, increasing in each place and has $\vep_0$ as its identity, where $\vep_0(0)=0$ and $\vep_0(x)=1$ for $x > 0$.

The class $\CT^*$ is the set of all binary operations $T^*$ on $[0,1]$ that are increasing in each place and have $0$ as identity. The class $\CL$ is the set of all binary operations $L$ on $[0,\infty]$ that satisfy the following conditions:
\begin{enumerate}[label=(L\arabic*)]
\item $L$ is surjective, 
\item $L$ is increasing in each place, and 
\item $L$ is continuous except possibly at the points $(0,\infty)$ and $(\infty,0)$.
\end{enumerate}
 
For any $T^* \in \CT^*$ and $L \in \CL$, $\tauTSL$ is the map on $\Delta^+ \times \Delta^+$ whose value, for any $F, G \in \Delta^+$, is the function $\tauTSL(F, G)$ defined on $[0,\infty]$ by
\[\tauTSL(F, G)(x)=\inf_{L(u,v)=x}T^*(F(u), G(v)).\]

Schweizer and Sklar posed the following open problem (see \cite[Problem 7.9.6]{Schweizer1983}).

\begin{prob}
Find conditions on $T$ and $L$ that are both necessary and sufficient for $\tauTL$ to be a triangle function. Similarly, find necessary and sufficient conditions on $T$ and $L$ for $\tauTL$ to be continuous. Clearly, the same problem can be posed for $\tauTSL$ and $\rho_{_{C,L}}$.
\end{prob}

The part of this problem concerning $\tauTL$ was solved in \cite{Lai2026}. Now we investigate the corresponding problem for $\tauTSL$ and $\rho_{_{C,L}}$, where $C$ is a copula. As observed by Saminger-Platz and Sempi \cite[Section~9, Equation~(9.1)]{SamingerPlatz2008}, the operation $T^*$ defined by $T^*(a,b)=a+b-C(a,b)$ for $a,b\in[0,1]$ belongs to $\CT^*$ and satisfies $\rho_{_{C,L}}=\tau_{_{T^*,L}}$. Thus, it suffices to characterize all pairs $(T^*, L)$ for which $\tau_{_{T^*,L}}$ is a (continuous) triangle function.

To rigorously establish the connection between $\Delta^+$ and the tensor product $\mathbb{E}$, we must account for the boundary continuity at infinity. Since functions in $\Delta^+$ are only required to be left-continuous on $]0, \infty[$, their left limit as $x \ra \infty$ may be strictly less than $1$. To accommodate this, let $\mathbb{E}'$ denote the space of functions $g\colon [0, 1] \to [0, 1]$ such that 
\begin{itemize}
\item $g(0)=1$, $g(1)=0$, 
\item $g$ is decreasing, and
\item $g$ is left-continuous on the open interval $]0, 1[$.
\end{itemize}

Let $\varphi\colon [0, \infty] \to [0, 1]$ be the order-preserving homeomorphism given by $\varphi(x) = 1 - e^{-x}$ for $x \in [0, \infty[$ and $\varphi(\infty) = 1$. On the ambient function spaces, define
\[
\widehat{\Phi}\colon [0,1]^{[0,\infty]}\to [0,1]^{[0,1]},
\qquad \widehat{\Phi}(H)(x)=1-H(\varphi^{-1}(x)).
\]
Its restriction $\Phi=\widehat{\Phi}|_{\Delta^+}$ maps $\Delta^+$ to $\mathbb{E}'$. Explicitly, for any $F \in \Delta^+$,
\[ \Phi(F)(x) = 1 - F(\varphi^{-1}(x)) =\begin{cases}
 1 - F(-\ln(1-x)) & \text{if}\ x \in [0, 1[,\\
 0 & \text{if}\ x=1.
\end{cases}\]
Because $F \in \Delta^+$ is increasing and left-continuous on $]0, \infty[$, and $\varphi^{-1}$ is strictly increasing and continuous on $[0, 1[$, we see that $\Phi(F)$ is decreasing and left-continuous on $]0, 1[$. Furthermore, the boundary condition $F(0) = 0$ naturally ensures that $\Phi(F)(0) = 1$. Thus, $\Phi$ is an order-reversing isomorphism between $\Delta^+$ and $\mathbb{E}'$.

For any $T^* \in \CT^*$ and $L \in \CL$, we define a corresponding operation $S$ on $[0,1]$ by $S(a,b) = \varphi(L(\varphi^{-1}(a), \varphi^{-1}(b)))$, and a dual operation $T$ on $[0,1]$ by $T(a,b) = 1 - T^*(1-a, 1-b)$. It is easy to observe that $T \in \CT$ and $S \in \CS$ (indeed, the induced operation $S$ inherently possesses even stronger boundary continuity than minimally required by the class $\CS$). Direct evaluation in the ambient function spaces gives
\[ \widehat{\Phi}(\tauTSL(F, G))(x) = (\Phi(F) \otimes_{_{S,T}} \Phi(G))(x). \]
Since $\Phi$ is an order-reversing isomorphism between $\Delta^+$ and $\mathbb{E}'$, this identity shows that $\tauTSL$ is a triangle function on $\Delp$ if and only if the induced operation $\otimes_{_{S,T}}$ is a t-conorm on $\mathbb{E}'$.

\begin{lem} \label{lem:tconorm_on_Eprime}
Let $S$ be a continuous t-conorm on $[0,1]$ and let $T$ be a right-continuous t-norm on $[0,1]$. If
\[
S(a,b)=1 \implies a=1\ \text{or}\ b=1,
\]
then $\otimes_{_{S,T}}$ is a t-conorm on $\mathbb{E}'$.
\end{lem}

\begin{proof}
Let $f,g\in\mathbb{E}'$. The arguments in Proposition \ref{prop:tensorST_decreasing} show that $f\otimes_{_{S,T}}g$ is decreasing and has value $1$ at $0$. The displayed condition gives
\[
(f\otimes_{_{S,T}}g)(1)=0,
\]
because every pair $(a,b)$ with $S(a,b)=1$ has $a=1$ or $b=1$, and a t-norm has $0$ as an annihilator.

It remains to verify left-continuity on $]0,1[$. Fix $x^*\in]0,1[$. Every fibre $K_x=\{(u,v)\mid S(u,v)=x\}$ with $x\leq x^*$ is contained in $[0,x^*]^2$. On this square, the functions $f$ and $g$ have exactly the same decreasing and left-continuous behaviour used in Lemma \ref{lem:supremum_attained} and Proposition \ref{prop:binary_operation}. Hence the proof of Proposition \ref{prop:binary_operation}, restricted to $[0,x^*]^2$, shows that $f\otimes_{_{S,T}}g$ is left-continuous at $x^*$. Thus $f\otimes_{_{S,T}}g\in\mathbb{E}'$.

Finally, $\omega_0\in\mathbb{E}'$. The identity and monotonicity computations in Propositions \ref{prop:tensorST_decreasing} and \ref{prop:is_tconorm} apply on $[0,1[$ without change. At $x=1$, both $f\otimes_{_{S,T}}\omega_0$ and $\omega_0\otimes_{_{S,T}}f$ have value $0=f(1)$ because they belong to $\mathbb{E}'$, so the identity law also holds at the endpoint. The associativity calculation applies on $[0,1[$; for any $f,g,h\in\mathbb{E}'$, at $x=1$ both sides are zero because $f\otimes_{_{S,T}}g$ and $g\otimes_{_{S,T}}h$ belong to $\mathbb{E}'$. Therefore $\otimes_{_{S,T}}$ is a t-conorm on $\mathbb{E}'$.
\end{proof}

\begin{thm} \label{thm:triangle_function_characterization} 
Given $T^* \in \CT^*$ and $L \in \CL$, the map $\tauTSL$ is a triangle function on $\Delta^+$ if and only if the following conditions hold:
\begin{enumerate}[label={\rm(\alph*)}]
\item \label{thm:triangle_function_characterization:L} $L$ is a continuous t-conorm on $[0, \infty]$ that has no zero divisor, in the sense that $L(u,v) = \infty$ implies $u = \infty$ or $v = \infty$;
\item \label{thm:triangle_function_characterization:T*} $T^*$ is a left-continuous t-conorm on $[0,1]$ if $L \neq \max$.
\end{enumerate}
\end{thm}

\begin{proof}
Under the isomorphism $\Phi$, $\tauTSL$ is a triangle function on $\Delta^+$ if and only if $\otimes_{_{S,T}}$ is a well-defined t-conorm on $\mathbb{E}'$.

Suppose first that $\otimes_{_{S,T}}$ is a t-conorm on $\mathbb{E}'$. If $S(a_0,b_0)=1$ for some $a_0,b_0<1$, then $\omega_{a_0},\omega_{b_0}\in\mathbb{E}'$, whereas
\[
(\omega_{a_0}\otimes_{_{S,T}}\omega_{b_0})(1)\geq T(1,1)=1.
\]
This contradicts the defining boundary condition at $1$. Hence
\[
S(a,b)=1 \implies a=1\ \text{or}\ b=1. \tag{B}
\]
For $u,v<1$, the functions $\omega_u,\omega_v$ belong to $\mathbb{E}'$, and (B) gives $S(u,v)<1$. The proof of Proposition \ref{prop:s_is_continuous_tconorm} therefore applies to these test functions and yields
\[
\omega_u\otimes_{_{S,T}}\omega_v=\omega_{S(u,v)}.
\]
The identity, commutativity, and associativity of $\otimes_{_{S,T}}$ imply the corresponding laws for $S$ on $[0,1[$. Monotonicity then extends the identity law to the endpoint and gives $S(1,t)=S(t,1)=1$ for every $t\in[0,1]$. Commutativity at an endpoint follows from these absorbing laws, and if any argument in the associativity law equals $1$, then both sides equal $1$; hence commutativity and associativity hold on all of $[0,1]$. Thus $S$ is a t-conorm. Since $S\in\CS$, Lemma \ref{lem:s1_implies_cont} shows that it is continuous on $[0,1]^2$.

If $S\neq\max$, the test functions used in Propositions \ref{prop:t_is_tnorm} and \ref{prop:t_is_rightcont} belong to $\mathbb{E}'$; moreover, every left-continuity argument in those proofs occurs at a point strictly below $1$. Applying the same calculations in $\mathbb{E}'$ shows that $T$ is a right-continuous t-norm. Under the homeomorphism $\varphi$ and the duality between $T$ and $T^*$, these conclusions say precisely that $L$ is a continuous t-conorm, that (B) is equivalent to the absence of zero divisors for $L$, and that $T^*$ is a left-continuous t-conorm when $L\neq\max$.

Conversely, assume conditions \ref{thm:triangle_function_characterization:L} and \ref{thm:triangle_function_characterization:T*}. If $L=\max$, then
\[
\tauTSL(F,G)(x)=\min\{F(x),G(x)\}
\]
for all $F,G\in\Delta^+$ and all $x\in[0,\infty]$, so $\tauTSL$ is a triangle function. Assume $L\neq\max$. The induced operation $S$ is a continuous t-conorm satisfying (B), and the dual operation $T$ is a right-continuous t-norm. Lemma \ref{lem:tconorm_on_Eprime} shows that $\otimes_{_{S,T}}$ is a t-conorm on $\mathbb{E}'$. Transporting this conclusion along $\Phi$ proves that $\tauTSL$ is a triangle function.
\end{proof}

\begin{thm} \label{thm:continuous_triangle_function_characterization}
Given $T^* \in \CT^*$ and $L \in \CL$, the map $\tauTSL$ is a continuous triangle function on $(\Delp, d_L)$ if and only if:
\begin{enumerate}[label={\rm(\alph*)}]
\item $L$ is a continuous t-conorm that has no zero divisor;
\item $T^*$ is a continuous t-conorm on $[0,1]$ if $L \neq \max$.
\end{enumerate}
\end{thm}

\begin{proof}
    Assume first that condition \ref{thm:triangle_function_characterization:L} holds and that $T^*$ is a continuous t-conorm when $L\neq\max$. If $L=\max$, then $\tauTSL(F,G)=\min\{F,G\}$. Moreover, for all $F,G,F',G'\in\Delta^+$, the pointwise minimum is nonexpansive for the modified L\'{e}vy metric:
    \[
        d_L(\min\{F,G\},\min\{F',G'\})
        \leq \max\{d_L(F,F'),d_L(G,G')\}.
    \]
    Hence $\tauTSL$ is a continuous triangle function in this case. Assume $L\neq\max$. By Theorem \ref{thm:triangle_function_characterization}, $\tauTSL$ is a triangle function. Let $F_n\xrightarrow{w}F$ and $G_n\xrightarrow{w}G$, and write $H_n=\tauTSL(F_n,G_n)$ and $H=\tauTSL(F,G)$. For every finite $x$, the fibre
    \[
        K_x=\{(u,v)\in[0,\infty]^2\mid L(u,v)=x\}
    \]
    is compact: it is closed by the continuity of $L$ and is contained in $[0,x]^2$ because $L$ is a t-conorm. Fix $x\in]0,\infty[$. For the lower estimate, choose $(u_n,v_n)\in K_x$ with
    \[
        T^*(F_n(u_n),G_n(v_n))\leq H_n(x)+\frac{1}{n}.
    \]
    Pass to a subsequence along which $H_n(x)$ converges to its lower limit; compactness then gives $(u_n,v_n)\ra(u,v)\in K_x$. Weak convergence and the left-continuity of $F$ and $G$ imply
    \[
        \liminf_{n\ra\infty}F_n(u_n)\geq F(u)
        \quad\text{and}\quad
        \liminf_{n\ra\infty}G_n(v_n)\geq G(v).
    \]
    Indeed, if $0<u'<u$ is a continuity point of $F$, then $u_n\geq u'$ eventually and $F_n(u_n)\geq F_n(u')\ra F(u')$; letting $u'\uparrow u$ through continuity points and using the left-continuity of $F$ gives the first inequality (the case $u=0$ is immediate), and the argument for $G$ is the same.
    The monotonicity and continuity of $T^*$ therefore yield
    \[
        \liminf_{n\ra\infty}H_n(x)\geq T^*(F(u),G(v))\geq H(x).
    \]
    For the upper estimate, fix $z>x$ and $\eta>0$, and choose $(u,v)\in K_z$ such that
    \[
        T^*(F(u),G(v))<H(z)+\eta.
    \]
    Choose sequences $u_m\uparrow u$ and $v_m\uparrow v$, keeping a zero coordinate fixed when necessary, so that every positive $u_m$ and $v_m$ is a continuity point of the corresponding distribution function. By the continuity of $L$, we have $y_m:=L(u_m,v_m)>x$ for all sufficiently large $m$. For such $m$,
    \[
        H_n(x)\leq H_n(y_m)\leq T^*(F_n(u_m),G_n(v_m)).
    \]
    Letting first $n\ra\infty$ and then $m\ra\infty$ gives
    \[
        \limsup_{n\ra\infty}H_n(x)\leq T^*(F(u),G(v))<H(z)+\eta.
    \]
    Since $z>x$ and $\eta>0$ are arbitrary,
    \[
        H(x)\leq\liminf_{n\ra\infty}H_n(x)
        \leq\limsup_{n\ra\infty}H_n(x)\leq\lim_{z\ra x^+}H(z).
    \]
    Consequently, $H_n(x)\ra H(x)$ at every continuity point of $H$. Proposition \ref{prop:weakly_converge} now yields $H_n\xrightarrow{w}H$, so $\tauTSL$ is continuous.

    Conversely, suppose that $\tauTSL$ is a continuous triangle function. By Theorem \ref{thm:triangle_function_characterization}, condition \ref{thm:triangle_function_characterization:L} holds and $T^*$ is a left-continuous t-conorm whenever $L\neq\max$. It remains to show in the latter case that $T^*$ is right-continuous. Suppose, for the sake of contradiction, that $T^*$ is not right-continuous. Then there exist $x_0, y_0 \in [0,1[$ such that $\inf \, \{ T^* (x, y) \mid x > x_0, y > y_0 \} > T^* (x_0, y_0)$. Choose $c\in ]0,\infty[$ such that $L(c,c) > c$; such a $c$ exists, for otherwise monotonicity and the identity law would give $L(u,v)=\max\{u,v\}$ for all $u,v$, contradicting $L\neq\max$. Define $f,g \in \Delp$ as follows:
    \[
        f(x) = \begin{cases}
            0, & x = 0,\\
            x_0, & 0 < x \leq c,\\
            1, & x > c,
        \end{cases}
        \quad     
        g(x) = \begin{cases}
            0, & x = 0,\\
            y_0, & 0 < x \leq c,\\
            1, & x > c.
        \end{cases}
    \]
    For $0<x\leq c$, every pair $(u,v)$ with $L(u,v)=x$ has $u,v\leq c$, and the pairs with one coordinate equal to $0$ show that the infimum is $\min\{x_0,y_0\}$. For $c<x\leq L(c,c)$, a pair with both coordinates in $]0,c]$ exists by the continuity of $L$, and it gives the value $T^*(x_0,y_0)$; every other pair gives a value at least this large. Finally, if $x>L(c,c)$, each pair with $L(u,v)=x$ has $u>c$ or $v>c$, so the value is $1$. Hence
    \[
        \tauTSL (f,g)(x) =
    \begin{cases}
        0, & x=0, \\
        \min\,\{x_0,y_0\}, & 0 < x \leq c, \\
        T^* (x_0,y_0), & c < x \leq L(c,c),\\
        1, & x > L(c,c).
    \end{cases}
    \]
    
    For each $n\in \bbN$, define
    \[
        f_n (x) = \begin{cases}
            0, & x = 0,\\
            x_0 + (1 - x_0) (\frac{x}{c})^n, & 0 < x \leq c,\\
            1, & x > c,
        \end{cases}
\quad
        g_n (x) = \begin{cases}
            0, & x = 0,\\
            y_0 + (1 - y_0) (\frac{x}{c})^n, & 0 < x \leq c,\\
            1, & x > c.
        \end{cases}
    \]

    \begin{center}
    \begin{tikzpicture}[scale=1.5, font=\small]
      \draw[->] (-0.5,0) -- (4,0) node[right] {$x$};
      \draw[->] (0,-0.5) -- (0,1.5) node[above] {$y$};
  
      \node[below left] at (0,0) {$O$};
      \def\c{2}
      \def\xzero{0.3}
  
      \draw[dashed] (\c,0) node[below] {$c$} -- (\c,1);
      \draw[dashed] (0,1) node[left] {$1$} -- (3.5,1);
      
      \foreach \n in {1, 2, 3, 5} {
        \draw[thin,domain=0.01:\c,samples=100] 
          plot (\x, {\xzero + (1-\xzero)*pow(\x/\c,\n)});
        \draw[thin] (\c,1) -- (3.5,1);
      }
      
      \filldraw (0,0) circle (1.2pt);
      \foreach \y in {\xzero} {
        \draw (0,\y) circle (1.2pt);
      }
      \node[left] at (0, \xzero) {$x_0$};
      \node[right] at (3,1.3) {$f_n(x)$};
      \node[black,right] at (3,1.1) {$n=1,2,3,5$};
    \end{tikzpicture}
    \end{center}
    
    \begin{center}
    \begin{tikzpicture}[scale=1.5, font=\small]
      \draw[->] (-0.5,0) -- (4,0) node[right] {$x$};
      \draw[->] (0,-0.5) -- (0,1.5) node[above] {$y$};
  
      \node[below left] at (0,0) {$O$};
      \def\c{2}
      \def\yzero{0.6}
  
      \draw[dashed] (\c,0) node[below] {$c$} -- (\c,1);
      \draw[dashed] (0,1) node[left] {$1$} -- (3.5,1);
      
      \foreach \n in {1, 2, 3, 5} {
        \draw[thin,domain=0.01:\c,samples=100] 
          plot (\x, {\yzero + (1-\yzero)*pow(\x/\c,\n)});
        \draw[thin] (\c,1) -- (3.5,1);
      }
      
      \filldraw (0,0) circle (1.2pt);
      \foreach \y in {\yzero} {
        \draw (0,\y) circle (1.2pt);
      }
      \node[left] at (0, \yzero) {$y_0$};
      \node[right] at (3,1.3) {$g_n(x)$};
      \node[black,right] at (3,1.1) {$n=1,2,3,5$};
    \end{tikzpicture}
    \end{center}
    
    Clearly, $f_n \xrightarrow{w} f$ and $g_n \xrightarrow{w} g$. Now we show that the sequence $\tauTSL(f_n, g_n)$ does not converge to $\tauTSL(f, g)$ on the open interval $]c,L(c,c)[$. Given $x \in ]c, L(c,c)[$, suppose that $L(u,v) = x$. The pairs with a zero coordinate must be handled separately. If $u=0$, then $v=x>c$, and hence
    \[
        T^*(f_n(u),g_n(v))=T^*(0,1)=1.
    \]
    Similarly, if $v=0$, then $u=x>c$ and $T^*(f_n(u),g_n(v))=T^*(1,0)=1$. Assume now that $u,v>0$. They cannot be simultaneously greater than or equal to $c$, as that would imply $L(u,v) \geq L(c,c) > x$. If at least one of them is greater than or equal to $c$, say $u \geq c$, then $f_n(u)=1$, so $T^*(f_n(u),g_n(v))=1$. If both $u<c$ and $v<c$, then $f_n(u)>x_0$ and $g_n(v)>y_0$. In every case,
    \[
        T^*(f_n(u),g_n(v)) \geq \inf \, \{ T^* (\alpha, \beta) \mid \alpha > x_0, \beta > y_0 \}.
    \]
    
    Taking the infimum over all such pairs $(u,v)$, we have
    \[
        \tauTSL(f_n, g_n)(x) = \inf_{L(u,v) = x} T^* (f_n (u), g_n (v)) \geq \inf \, \{ T^* (\alpha, \beta) \mid \alpha > x_0, \beta > y_0 \}.
    \]
    Taking lower limits as $n\ra\infty$, we obtain
    \[\liminf\limits_{n \rightarrow \infty} \tauTSL(f_n, g_n)(x) \geq \inf \, \{ T^* (\alpha, \beta) \mid \alpha > x_0, \beta > y_0 \}> T^* (x_0, y_0)= \tauTSL (f,g)(x).\]
    Together with the preceding pointwise bound, this implies a strict separation between $\tauTSL(f_n,g_n)$ and $\tauTSL(f,g)$ on the interval $]c, L(c,c)[$ for all $n\in \mathbb{N}$. Since $\tauTSL(f,g)$ is constant on $]c,L(c,c)[$, every point of this interval is a continuity point of $\tauTSL(f,g)$. Therefore, by Proposition \ref{prop:weakly_converge}, the sequence $\tauTSL(f_n, g_n)$ fails to weakly converge to $\tauTSL(f, g)$, implying that $\tauTSL$ is not continuous. Thus, we arrive at a contradiction.
\end{proof}

\begin{rem}
Under the additional conditions (L0) and (LS) on $L$, the classical sufficiency result for the continuity of the infimum convolution $\tauTSL$ is given in \cite[Lemma~8.3]{SamingerPlatz2008}. For the supremum convolution $\tauTL$, \cite[Lemma~7.14]{SamingerPlatz2008} gives an analogous necessity result for continuity of the generating commutative semicopula under these conditions. Here (L0) means that $0$ is the identity of $L$, and (LS) requires $L(u_1,v_1)<L(u_2,v_2)$ whenever $u_1<u_2$ and $v_1<v_2$; see \cite[Definition~3.5]{SamingerPlatz2008}. Theorem \ref{thm:continuous_triangle_function_characterization} does not require (LS). The proof above therefore covers the more general class of continuous t-conorms with no zero divisors considered here.
\end{rem}

\section*{Acknowledgments}

The authors acknowledge support from the National Natural Science Foundation of China (Grant Nos. 12671089 and 12671561). The authors would like to thank Professor Hongliang Lai for helpful discussions.

\section*{Declaration of generative AI and AI-assisted technologies
in the manuscript preparation process}

During the preparation of this work, the authors used ChatGPT to assist with manuscript organization, language editing, consistency checking, and submission-readiness review. After using this tool, the authors reviewed and edited the content as needed and take full responsibility for the content of the published article.

\bibliographystyle{abbrv}
\bibliography{lili}

\end{document}